\documentclass[reqno]{amsart}
\usepackage{amsmath, amsthm, amssymb, enumitem, mathrsfs}
\usepackage{hyperref, soul}
\usepackage{verbatim} %For \begin{comment} ... \end{comment} command
\usepackage{tikz}
	\usetikzlibrary{cd}
\usepackage{tikz-cd}

\tikzset{%
	symbol/.style={%
		draw=none,
		every to/.append style={%
			edge node={node [sloped, allow upside down, auto=false]{$#1$}}}
	}
}
\usepackage{color}
\usepackage{cite}
\usepackage{bbm}
\usepackage{setspace}

\usetikzlibrary{calc}
\usetikzlibrary{decorations.pathmorphing}

\tikzset{curve/.style={settings={#1},to path={(\tikztostart)
    .. controls ($(\tikztostart)!\pv{pos}!(\tikztotarget)!\pv{height}!270:(\tikztotarget)$)
    and ($(\tikztostart)!1-\pv{pos}!(\tikztotarget)!\pv{height}!270:(\tikztotarget)$)
    .. (\tikztotarget)\tikztonodes}},
    settings/.code={\tikzset{quiver/.cd,#1}
        \def\pv##1{\pgfkeysvalueof{/tikz/quiver/##1}}},
    quiver/.cd,pos/.initial=0.35,height/.initial=0}

\setenumerate{label=\textnormal{(\arabic*)}} %Default enumeration is upright for enumitem

\newcommand{\nada}{\varnothing}
\renewcommand{\epsilon}{\varepsilon}
\renewcommand{\phi}{\varphi}

\newcommand{\A}{\mathcal{A}}
\newcommand{\C}{\mathcal{C}}

\newcommand{\D}{\mathcal{D}}
\newcommand{\G}{\mathcal{G}}

\renewcommand{\L}{\mathcal{L}}
\newcommand{\M}{\mathcal{M}}

\newcommand{\Tau}{\mathcal{T}}
\newcommand{\U}{\mathcal{U}}
\newcommand{\V}{\mathcal{V}}
\newcommand{\W}{\mathcal{W}}
\newcommand{\X}{\mathcal{X}}

\newcommand{\Z}{\mathbb{Z}}

\newcommand{\id}{\text{id}}
\newcommand{\cat}{\mathsf}
\newcommand{\mono}{\rightarrowtail}
\renewcommand{\hom}{\text{Hom}}

\newcommand{\vcat}{\V\text{-}\cat{Cat}}
\newcommand{\ucat}{\U\text{-}\cat{Cat}}

\newcommand{\bcg}{\widetilde{G}}

\numberwithin{equation}{section}

\theoremstyle{plain}
\newtheorem{theorem}[equation]{Theorem}
\newtheorem{corollary}[equation]{Corollary}
\newtheorem{lemma}[equation]{Lemma}
\newtheorem{proposition}[equation]{Proposition}

\theoremstyle{definition}
\newtheorem{definition}[equation]{Definition}

\newtheorem{example}[equation]{Example}
 
\newtheorem{remark}[equation]{Remark}
\newtheorem{hypothesis}[equation]{Hypothesis}

\begin{document}

\begin{spacing}{1.05}

\title{Enriched coverages and sheaves under change of base}

\author{Ariel E. Rosenfield}
\email{rosenfia@uci.edu}
\address{Department of Mathematics\\ University of California, Irvine \\
440R Rowland Hall\\ Irvine, CA 96267--3875 \\ USA}

\date{\today}
%\keywords{noncommutative ring, category, functor}
%\subjclass[2010]{
%Primary: 
% 14A22, %Noncommutative algebraic geometry
%16B50; %  Category-theoretic methods and results
%Secondary: 
%16P40% Noetherian rings and modules 
%}

\begin{abstract} We investigate how change of base of enrichment via a faithful, conservative right adjoint functor interacts with enriched coverages and sheaves on a given enriched category. We prove that change of base via such a functor gives rise both to an injective mapping on subobjects in enriched presheaf categories, and to an injective mapping on enriched coverages. In case the base change functor is also full, the enriched associated sheaf construction on a presheaf category commutes with base change.
\end{abstract}

\maketitle

\setcounter{tocdepth}{2} %Set value to 2 to include subsections
\makeatletter
\def\l@subsection{\@tocline{2}{0pt}{2.5pc}{5pc}{}}
\def\l@subsubsection{\@tocline{2}{0pt}{5pc}{7.5pc}{}} %Code to indent subsections for AMS packages
\makeatother
\tableofcontents

\section{Introduction}

As outgrowths of the move to formalize algebraic geometry in terms of abelian categories, Grothendieck topologies and their accompanying categories of sheaves arose in the early 1960s as a framework for defining cohomology theories on schemes. Roughly speaking, a Grothendieck topology on a category $\C$ can be regarded as a way to specify, for all objects $U$ of $\C$, which objects of $\C$ cover $U$. This is in exactly the same sense as, given a topological space $X$ and an open set $U \subset X$, we might ask when $\bigcup_{i \in I} U_i = U$ for some family $\{U_i : i \in I\}$ of opens of $X$. Enriched categories, where the hom-sets of ordinary category theory are replaced, more generally, by objects of a closed monoidal category $\V$, were first introduced in the mid-1960s in the work of Maranda \cite{Maranda_1965} and B\'{e}nabou \cite{benabou-enriched}, among others. Around the same time, Gabriel introduced in \cite[V.2, p. 411]{gabriel} the notion of a (right) linear topology (\textit{topologie lin\'{e}aire \`{a} droite}) on a ring, an early example of an enriched Grothendieck topology in the particular case of a category with one object enriched over $\V = \cat{Ab}$.

The definition of a Grothendieck topology in terms of sieves on objects $x \in \C$ (that is, subobjects of $\C(-,x)$) is perhaps the most straightforwardly generalizable to the enriched setting. For a nice enough base category $\V$, enriched Grothendieck topologies on a $\V$-category $\C$ (now taken to be families of subfunctors of enriched hom-functors), their accompanying sheaves, and their correspondence with localizations of $[\C^\text{op},\V]$, were introduced by Borceux and Quinteiro in 1996 with the publication of \cite{bq}. More recently, details of the theory of enriched sheaves in the case $\V = \cat{Ab}$ were established in the 2000s by Lowen in \cite{Lowen_2004} and \cite{Lowen_2016}; and in 2020 by Coulembier \cite{Coulembier}.

Given a category $\C$ enriched over $(\V, \otimes, *_\V)$ and a lax monoidal functor $G : \V \to \U$, $G$ canonically induces a 2-functor $$G_* : \V\text{-}\cat{Cat} \to \U\text{-}\cat{Cat}$$ which acts via an operation called `base change' or `change of base.' Base change first appeared in the literature around the same time as enriched categories themselves, with Eilenberg and Kelly's publication of \cite{Eilenberg_Kelly_1966}, and is fundamental to the theory of enriched categories, in part because it allows one to view a $\V$-category $\C$ as an ordinary category by applying the functor $$ \hom_\V(*_\V,-) : \V \to \cat{Set} $$ to the hom-objects of $\C$. Many of the technical results in Section 3 of the current work rely on the results and style of argument developed in Cruttwell's 2008 doctoral thesis \cite{cruttwell}, which, toward understanding normed spaces, addressed in detail the question of how base change interacts with the monoidal structures on $\V$ and $\U$.

A central theme of this work is the following: Changing base via a particular $G$ may result in more or less loss of information about the hom-objects of $\C$. To illustrate, we consider the functors $$ \hom_\cat{Ab}(\Z,-) : \cat{Ab} \to \cat{Set} \quad \text{ and } \quad \hom_{\cat{grMod}_k}(k,-) : \cat{grMod}_k \to \cat{Set}, $$ where $k$ is a field, and gradings are taken over $\Z$. Letting $\V$ be either of $\cat{Ab}$ or $\cat{grMod}_k$, we define the hom-objects of the $\cat{Set}$-category $G_*\C$ to be $$ G_*\C(x,y) := G(\C(x,y)). $$ In the former case, the hom-sets resulting from base change are in bijection with the underlying sets of the original hom-objects, and the $\U$-topology resulting from changing the base of a $\V$-topology is no coarser than the one we started with. In the latter case, however, for a graded $k$-module $M := \C(x,y)$, we only recover the set $$\hom_\V(k,M) \cong \hom_k(k, M_0) \cong M_0$$ of degree-0 elements of $M$ after changing base---in this case, the $\U$-topology resulting from a given $\V$-topology is much coarser. The key difference between these two examples lies in whether or not $\hom_\V(*_\V,-)$ is faithful; or equivalently, whether $\{*_\V\}$ is a dense generating family for $\V$.

Below, we examine situations where this `loss' is minimal, and situations where changing base results in topologies which are coarser than the ones we started with, as in \ref{bad-gabriel}.

\subsection{Acknowledgements} This work constitutes a portion of the author's Ph.D. thesis. Thanks to Manny Reyes for all his support, and to So Nakamura and Cody Morrin for their comments and questions on early drafts. Thanks also to Ana Luiza Ten\'orio for discussions leading to Example \ref{example:preorder-gt}.

\section{Preliminaries}

Let $(\V, \otimes, I)$ denote a symmetric monoidal closed category. By a \textbf{small} $\V$-category, we mean one which is equivalent to a $\V$-category with a small set of objects. We will always use $\C$ to denote a small $\V$-category. 

To ensure continuity of this work with that of Borceux-Quinteiro \cite{bq} and Kelly \cite{kelly-structures}, we make the following assumptions on $\V$:

\begin{hypothesis} \label{assumptions-V} Unless otherwise indicated, 

    \begin{enumerate}

    \item[(i)] $\V$ is locally finitely presentable (so $\V$ is well-powered, and admits a dense generating set $\G_\V$ of finitely presentable objects);

    \item[(ii)] as an unenriched category, $\V$ is locally small;

    \item[(iii)] as a self-enriched category, $\V$ admits all small conical limits and colimits (so is tensored and cotensored over itself);

    \item[(iv)] a finite tensor product of finitely presentable objects of $\V$ is again finitely presentable;

    \item[(v)] $\V$ is regular in the sense of \cite{barr}.
    
\end{enumerate}
\end{hypothesis} 

Examples of categories which satisfy these conditions include: $\cat{Set}$, $\cat{Ab}$; $\cat{Cat}$ and $\cat{sSet}$; the categories $\cat{Mod}_k$ of $k$-modules, $\cat{grMod}_k$ of graded $k$-modules, and $\cat{Coalg}_k$ of $k$-coalgebras for $k$ a commutative ring.

\subsection{Change of base} A very detailed treatment of this topic can be found in \cite[4]{cruttwell}, but for convenience, we recount the bare rudiments here. Let $$ (\mathcal{U}, \otimes, *_\U) \quad \text{and} \quad (\V, \otimes, *_\V) $$ be categories satisfying Hypothesis \ref{assumptions-V}. (We will never need to distinguish between the two monoidal operations, so we use the same symbol for both.) Denote an identity morphism in an enriched category $\X$ by $\text{id}^\X$, and a composition morphism in $\X$ by $\circ^\X$. For visual simplicity, we will often omit subscripts which would ordinarily indicate the domain objects of the morphisms $\text{id}$ and $\circ$. 

First, we recall the definition of the underlying category of $\C$, which we take to be a small $\V$-category, as usual.

\begin{definition} \label{underlying-category} \textbf{(Underlying category construction.)} Define an ordinary category $\C_0$ by setting $\text{Ob}(\C_0) = \text{Ob}(\C)$ and $\C_0(x,y) = \hom_\V(*_\V,\C(x,y))$. Given morphisms $g : x \to y$ and $f : y \to z$ in $\C_0$, we define the composite $f \cdot g$ by \[\begin{tikzcd}
	{*_\V} & {*_\V \otimes *_\V} & {\C(y,z) \otimes \C(x,y)} & {\C(x,z)}
	\arrow["\sim", from=1-1, to=1-2]
	\arrow["{f \otimes g}", from=1-2, to=1-3]
	\arrow["{\circ^\C}", from=1-3, to=1-4]
\end{tikzcd}.\]
\end{definition}

We will also make use of pre- and post-composition by morphisms in the underlying category, so we recall the definition:

\begin{definition} \label{def:pre-post-composition}
Given a morphism $g : *_\V \to \C(y,z)$ in $\V$, define
$$ g^* : \C(z,x) \cong \C(z,x) \otimes *_\V \xrightarrow{\id \otimes g} \C(z,x) \otimes \C(y,z) \xrightarrow{\circ^\C} \C(y,x)$$
and
$$ g_* : \C(x,y) \cong *_\V \otimes \C(x,y) \xrightarrow{g \otimes \id} \C(y,z) \otimes \C(x,y) \xrightarrow{\circ^\C} \C(x,z).$$
\end{definition}

Given a lax monoidal functor $G : \V \to \mathcal{U}$ with coherence morphisms $$ u : *_\V \to G(*_\U), \quad m_{xy} : G( x) \otimes G(y) \to G(x \otimes y),$$ we have a 2-functor ${G_* : \vcat \to \ucat}$. For notational convenience, we unpack this as follows:

\begin{definition} \label{cb-functor} With $G$ as above,

    \begin{enumerate}
        \item[(i)] form a $\mathcal{U}$-category $G_*\C$ by setting \begin{align*} \text{Ob}(G_*\C) & := \text{Ob}(\C), \\ G_*\C(x,y) &:= G(\C(x,y)), \\ \text{id}^{G_*\C} &:= G(\text{id}^\C) \cdot u  \\ \circ^{G_*\C} &:= G(\circ^\C) \cdot m . \end{align*}
    
        \item[(ii)] For a $\V$-functor $A : \C \to \D$, let \[ G_*A : G_*\C \to G_*\D\] denote the $\mathcal{U}$-functor defined by \[ G_*Ax := Ax \; \text{ and } \; (G_*A)_{xy} := GA_{xy} : G(\C(x,y)) \to G(\D(Ax, Ay)).\]

        \item[(iii)] For a $\V$-natural transformation $$ \{\alpha_x : *_\V \to \D(Ax, Bx)\},$$ let $G_*\alpha$ denote the $\mathcal{U}$-natural transformation $$ \{ G(\alpha_x) \cdot u : *_\U \to G(\D(Ax, Bx)) \} .$$ 
    \end{enumerate}
    
\end{definition}

We briefly note that Definition \ref{underlying-category} is a special case of base change using the monoidal functor $\hom_\V(*_\V,-) : \V \to \cat{Set}$.

\subsection{Four views on a monoidal adjunction} For the remainder of the work, we will be concerned almost exclusively with the case where the functor $G : \U \to \V$ is a right adjoint in the 2-category $\cat{MonCat}_\ell$ of monoidal categories and lax monoidal functors. Establishing notation, we suppose given a monoidal adjunction \begin{equation} \label{main-guy}
\begin{tikzcd}
	\U \arrow[r,bend left,"F",""{name=A, below}] & \V \arrow[l,bend left,"G",""{name=B,above}] \arrow[from=A, to=B, symbol=\dashv]
\end{tikzcd},
\end{equation} whose unit and counit we denote respectively by $\eta$ and $\epsilon$. Recall that by \cite[1.4]{doctrinal-adjunction}, the left adjoint of a monoidal adjunction is necessarily strong monoidal. Recall also that by \cite[2.1.3]{Riehl_Verity_2022}, the monoidal adjunction \ref{main-guy} induces a 2-adjunction \[
\begin{tikzcd}
	\U\text{-}\cat{Cat} \arrow[r,bend left,"F_*",""{name=A, below}] & \V\text{-}\cat{Cat} \arrow[l,bend left,"G_*",""{name=B,above}] \arrow[from=A, to=B, symbol=\dashv]
\end{tikzcd}.
\] 

Using $\W$ to denote either of the categories $\U$ or $\V$, we discuss three perspectives on $\W$, which we will need in order to prove the results of \textsection 3 in full detail. In cases where we need to distinguish between $\W$ viewed as self-enriched and as an unenriched category, we will use the notation $\W$ for the enriched category and $\W^u$ for the unenriched category. 

It is proved in \cite[3.4.9]{riehl-cht} that the monoidal closed structure of $\V^u$ canonically induces an isomorphism $\V_0 \cong \V^u$; thus, we may view $\V$ as a $\V$-category, as an unenriched category in its own right, or as an unenriched category $\V_0$. We will see next that the categories $\V_0$ and $(G_*\V)_0$ are also isomorphic.  

\begin{proposition} \label{prop:c0-gc0-iso} For a $\V$-category $\C$, the categories $\C_0$ and $(G_*\C)_0$ are isomorphic. 
\end{proposition}

\begin{proof} Define $k_{x,y}$ to be the composite bijection $$ \hom_{\V^u}(*_\V, \C(x,y)) \to \hom_{\V^u}(F(*_\U), \C(x,y)) \to \hom_{\U^u}(*_\U, G\C(x,y)) $$ arising from the isomorphism $*_\V \cong F(*_\U)$ and the adjunction \ref{main-guy}. Note that by coherence for \ref{main-guy}, the action of $k$ on elements of $\hom_\V(*_\V, \C(x,y))$ is $ f \mapsto G(f) \circ u$. To see that $k$ preserves composites, note that coherence of $G$ ensures that each region of the diagram \[\begin{tikzcd}
	{*_\U} & {G*_\V} & {G(*_\V \otimes *_\V)} & {G(\C(y,z) \otimes \C(x,y))} & {G(\C(x,z))} \\
	{*_\U \otimes *_\U} && {G(*_\V) \otimes G(*_\V)} & {G(\C(y,z)) \otimes G(\C(x,y))}
	\arrow["u", from=1-1, to=1-2]
	\arrow["\sim", from=1-1, to=2-1]
	\arrow["\sim", from=1-2, to=1-3]
	\arrow["{G(f \otimes g)}", from=1-3, to=1-4]
	\arrow["{G\circ}", from=1-4, to=1-5]
	\arrow["{u \otimes u}", from=2-1, to=2-3]
	\arrow["m", from=2-3, to=1-3]
	\arrow["{Gf \otimes Gg}", from=2-3, to=2-4]
	\arrow["m", from=2-4, to=1-4]
	\arrow["\circ"'{pos=0.6}, from=2-4, to=1-5]
\end{tikzcd}\] commutes, which says exactly (recalling Definition \ref{underlying-category}) that ${k(f \cdot g) = k(f) \cdot k(g)}$. That $k$ preserves identities is easily shown using the unit coherence axiom for $G$.
\end{proof} 

For the sake of the discussion below, we use the notation \begin{equation} \label{eq:i-j-k}
    i : \V_0 \to \V^u, \; j: \U_0 \to \U^u, \text{ and } k: \V_0 \to (G_*\V)_0
\end{equation} for the isomorphisms canonically induced by the closed monoidal structure of $\V$ and $\U$ and (in the case of $k$) by monoidal coherence of $G$. To make our notation consistent, recalling Definition \ref{cb-functor}, observe that $$ (G_*\alpha)_x = k(\alpha_x) $$ for a given $\V$-natural transformation $\alpha : A \Rightarrow B$.

The monoidal adjunction \ref{main-guy} induces both an unenriched adjunction \begin{equation} \label{eq:induced-unenriched-adjunction}
    (i^{-1}Fj) : \U_0 \leftrightarrows \V_0 : (j^{-1}Gi)
\end{equation} and, via the argument in \cite[1.11]{kelly}, a $\U$-adjunction \begin{equation} \label{eq:induced-U-adjunction}
    F^\U : \U \leftrightarrows G_*\V : G^\U.
\end{equation} The action of the right adjoint $G^\U$ of this pair on hom-objects is defined by letting $G^\U_{xy} : G(\V(x,y)) \to \U(Gx,Gy)$ be the morphism in $\U^u$ corresponding by the Yoneda lemma to the map \[\begin{tikzcd}
	{\hom_{(G_*\V)_0}(x,y)} && {\hom_{\U_0}(Gx, Gy)}
	\arrow["{j^{-1}G_{xy}ik^{-1}}", from=1-1, to=1-3]
\end{tikzcd}.\] As such, we have $j^{-1}G_{xy}i = (G^\U_{xy})_* \circ k$.

Denoting the components of the associated natural isomorphism of hom-objects by $$ \Phi^\U : G\V(Fx,y) \longrightarrow \U(x,Gy),$$ we have $G^\U_{xy} = \Phi^\U_{(Gx)y} \circ k(\epsilon_x)^*$ by the triangle identities for \ref{eq:induced-U-adjunction}. This proves the following, which we state as a remark for future reference.

\begin{remark} \label{remark:G0-is-GU} For arbitrary $x,y \in \V$, we have an equality of functions $$ \Phi^\U_{(Gx)y} \circ k(\epsilon_x)^* \circ k = j^{-1}G_{xy} i : \hom_{\V_0}(x,y) \longrightarrow \hom_{\U_0}(Gx,Gy). $$
\end{remark}

\subsection{$\V$-limits} The enriched limits we encounter in this work are as simple as possible; namely, enriched limits weighted by a constant functor. We recall the definition for convenience.

\begin{definition}
    Let $* : \mathcal{D} \to \V$ be an ordinary functor constant at the monoidal unit $*_\V$ of $\V$, and let $F : \mathcal{D} \to \mathcal{C}$ be a $\V$-functor. The \textbf{conical limit} of $F$, if it exists, is an object $\lim^* F$ of $\C$ defined by the universal property $$ \C(m, \text{lim}^* F) \cong [\mathcal{D},\mathcal{V}](*,\C(m,F(-))). $$
\end{definition} 

In particular, we will often make use of cotensors, which we recall are a special case of conical limits (see \cite[7.4.3]{riehl-cht}). Finally, we make the following remark, to be used later in the work.

\begin{remark} \label{remark:conicallimitsareordinarylimits}
    In the setting of Hypothesis \ref{assumptions-V}, conical limits in $\V$ coincide with ordinary limits in $\V_0$, as observed in \cite[p. 50]{kelly}. 
\end{remark}

\subsection{Enriched coverages} Here, we recall \cite[Def. 1.2]{bq}, and lay out a few details for notational clarity in the enriched setting. In particular, to give the statement clearly, we will need to understand monomorphisms, subobjects, cotensors, and pullbacks in the category $[\C^\text{op},\V]$.

By a \textbf{monomorphism} $\eta: F \Rightarrow G$ between $\V$-functors $F,G : \C^\text{op} \to \V$, we mean a $\V$-natural transformation whose components are each monomorphisms in $\V_0$. For a $\V$-functor $K \in [\C^\text{op},\V]$, denote \begin{equation} \label{def:mono-lattice} \M_\V(K) := \{ \alpha \in \text{Mor}([\C^\text{op},\V]_0): \alpha \text{ monic and } \text{cod}(\alpha) = K\}.\end{equation} By a \textbf{subobject} of $F$, we mean an equivalence class in $ \M_\V(K)$ under the relation $$ (r : R \mono K) \sim (s : S \mono K) \; \iff \; r = st \text{ for some isomorphism } t.$$ We denote the set of such equivalence classes by \begin{equation} \label{def:sub-lattice} \text{Sub}_\V(K) := \M_\V(K) / \sim. \end{equation} Sieves in the $\V$-enriched setting are then defined exactly as in the unenriched case.

\begin{definition} \label{v-sieve}
    Let $\C$ be a $\V$-category, and let $x \in \C$ be an object. A \textbf{sieve} on $x \in \C$ is a subobject of $\C(-,x) \in [\C^\text{op},\V]$.
\end{definition}

The enriched functor category $[\C^\text{op},\V]$ has all small conical limits and colimits if $\V$ does, as explained in \cite[3.3]{kelly}. Thus $[\C^\text{op},\V]$ is cotensored over $\V$:

\begin{definition} \label{functor-cotensor} The \textbf{cotensor} $\{v,B\}$ of $B \in [\C^\text{op},\V]$ by $v \in \V$ is the $\V$-functor whose action on an object $x \in \C$ is $\V(v, Bx) \in \V$, together with $\V$-natural isomorphisms \[ [\C^\text{op},\V](A, \{v,B\}) \cong \V(v, [\C^\text{op},\V](A,B)).\]
\end{definition}

In the presence of the monoidal adjunction \ref{main-guy}, and assuming that $\C$ is cotensored over $\V$, change of base makes $G_*\C$ cotensored over $\U$ as follows:

\begin{remark} \label{cb-cotensor} Given a cotensored $\V$-category $\C$, $G_*\C$ is cotensored over $\U$ via $$ \{u, x\} := \{Fu, x\} $$ for objects $u \in \U$ and $x \in G_*\C$.
\end{remark}

Note that for any $v \in \V$, a monomorphism $R \mono \C(-,x)$ of $\V$-functors induces, by naturality of cotensoring, a monomorphism $$\{v , R\} \mono \{v, \C(-,x)\},$$ which we denote by $\iota$. Moreover, the enriched Yoneda lemma \cite[7.3.5]{riehl-cht} tells us that any $f : v \to \C(y,x)$ induces a map $v \to \cat{Nat}_\V(\C(-,y),\C(-,x))$, which in turn induces a $\V$-natural transformation $f: \C(-,y) \to \{v, \C(-,x)\}$. With this notation, we define the pullback of a sieve $R$ as follows:

\begin{definition} \label{pullback}
    The limit $R_f$ of the diagram $$ \begin{tikzcd}
	{\C(-,y)} && {\{v,\C(-,x)\}} && {\{v,R\}}
	\arrow["f", from=1-1, to=1-3]
	\arrow["\iota"', from=1-5, to=1-3]
\end{tikzcd} $$ in $[\C^\text{op},\V]$ is defined pointwise as the functor $\C^\text{op} \to \V$ whose value $R_f(z)$ at $z \in \C$ is the pullback of the diagram $$ \begin{tikzcd}
	{\C(z,y)} && {\V(v,\C(z,x))} && {\V(v,Rz)}
	\arrow["f_z", from=1-1, to=1-3]
	\arrow["\iota_z"', from=1-5, to=1-3]
\end{tikzcd} $$ in $\V$.
\end{definition}

Now we turn to the main definition of interest. We will often use a weakened form of the original definition \cite[1.2]{bq}, so our statement differs slightly from theirs. Recall that $\G_\V$ denotes a dense generating family of finitely presentable objects of $\V$. For $x,y \in \C$, we take the perspective that an element $f \in \hom_\V(\G_\V, \C(y,x))$ (i.e., a morphism $f : g \to \C(y,x)$ for some $g \in \G_\V$) is a generalized element of $\C(y,x)$.

\begin{definition} \label{v-gt} Given a small $\V$-category $\C$, let $J$ be an assignment to each object $x$ in $\C$ of a family $J(x)$ of sieves on $x$. The assignment $J$ may satisfy
	\begin{itemize}
		\item[(T1)] $\C(-,x) \in J(x)$ for each object $x$;
		
		\item[(T2)] for any $y \in \C$, any $R \in J(x)$, and any $f \in \hom_\V(\G_\V, \C(y,x))$, we have $R_f \in J(y)$, where $R_f$ is as in Definition \ref{pullback};
		
		\item[(T3)] if $R \mono \C(-,x)$ is an arbitrary sieve for which there exists $S \in J(x)$ such that for all objects $y$ of $\C$, $$R_f \in J(y) \text{ for any } f \in \hom_\V(\G_\V,S(y)),$$ then we have $R \in J(x)$.
	\end{itemize} We say that $J$ is a \textbf{$\V$-coverage} if it satisfies (T1) and (T2), and that $J$ is a \textbf{$\V$-Grothendieck topology}, or simply a \textbf{$\V$-topology}, if it satisfies (T1)-(T3).
\end{definition}

For clarity, we emphasize here that, since the collections $J(x)$ are in fact families of isomorphism classes of functors, if $R \in J(x)$ and $(s : S \mono \C(-,x))$ is $\V$-naturally isomorphic to $R$, then $S \in J(x)$. 

\subsection{Examples of $\V$-coverages} Finally, we give a few examples of enriched coverages and topologies to motivate the rest of the work.

Given a not-necessarily commutative ring $A$ viewed as a one-object $\cat{Ab}$-category, the standard example of \ref{v-gt}.ii comes from ring theory, where it is called a Gabriel topology \cite[VI.5]{Stenstrom_1975} on $A$. As an example of an enriched topology, it is noted both in \cite{bq} and by Lowen in \cite[2.4]{Lowen_2004}. 

\begin{example} \label{example:ring-gabriel} \textbf{(Gabriel topologies on a ring)} Let $A$ be a ring and let $\mathfrak{R}$ be a non-empty set of right ideals of $A$. The family $\mathfrak{R}$ is a Gabriel topology on $A$ if
    \begin{enumerate}
        \item[(R1)] $I \in \mathfrak{R}$ and $I \subset J$ implies $J \in \mathfrak{R}$;

        \item[(R2)] if $I \in \mathfrak{R}$ and $x \in A$, then $$(I : x) := \{ r \in A : xr \in I\} \in \mathfrak{R};$$

        \item[(R3)] if $I$ is a right ideal and there exists $J \in \mathfrak{R}$ such that $(I:x) \in \mathfrak{R}$ for every $x \in J$, then $I \in \mathfrak{R}$.
    \end{enumerate}

\end{example}

In Section 6 (Proposition \ref{v-topology-v-gabriel}), we define Gabriel topologies for a general $\V$, and give a detailed proof that a $\V$-topology on a one-object $\V$-category is the same thing as a $\V$-Gabriel topology, subsuming Example \ref{example:ring-gabriel}.

In the case where $\C$ is a preordered set---that is, a category enriched over the monoidal preorder $\V = \{0,1\}$---a sieve on $p \in \C$ is exactly a downward-closed subset of $\downarrow p$, and the pullback of a sieve $S$ on $p$ along a morphism $q \leq p$ is exactly $S \cap \downarrow q$ (note that this set is again a sieve on $q$). We obtain the following example:

\begin{example} \textbf{($\{0,1\}$-topologies on a preorder)} \label{example:preorder-gt} A $\{0,1\}$-Grothendieck topology $J$ on $\C$ is, to each $p \in \C$, a collection $J(p)$ of downward-closed subsets of $\downarrow p$ satisfying
    \begin{enumerate}
        \item[(P1)] the maximal sieve $\downarrow p$ is in $J(p)$;
        \item[(P2)] if $S \in J(p)$ and $q \leq p$, then $S \cap \downarrow q \in J(q)$;
        \item[(P3)] if $S \in J(p)$ and $R$ is a sieve on $p$ such that $R \cap \downarrow q \in J(q)$ for all $q \in S$, then $R \in J(p)$.
    \end{enumerate}
\end{example}

In the case where $\V$ is the monoidal preorder $([0,\infty], \geq, + , 0)$, $\V$-categories are the generalized metric spaces of \cite{Lawvere_1973}. Note that $[0,\infty]$ is not locally finitely presentable as a category; however, Definitions \ref{v-sieve} and \ref{v-gt} do not require the enriching category to be locally finitely presentable. To make sense of them, we could merely have asked that $\V$ admit some strongly generating set of objects---in $[0,\infty]$, we can take $\G_\V = \mathbb{Q} \cap [0,\infty]$. A $\V$-coverage on a generalized metric space is then as follows.

\begin{example} \textbf{($[0,\infty]$-coverages on a Lawvere metric space)} \label{example:lawvere-coverage} As described in \cite{Lawvere_1973}, a $[0,\infty]$-functor is a (1-)Lipschitz function between generalized metric spaces. Let $\mathcal{L}$ be a $\V$-category. Recall that $\V(x,y) := \max\{0, y-x\}$, and that $\V$ is generated under filtered colimits by $\mathbb{Q} \cap [0,\infty]$. To emphasize that we are viewing $\mathcal{L}$ as a generalized metric space, we denote $ d(x,y) := \mathcal{L}(x,y)$. 

A $\V$-sieve on $x \in \mathcal{L}$ is a 1-Lipschitz function $r : \mathcal{L} \to [0,\infty]$ which satisfies (for all $y,z \in \mathcal{L}$)
\begin{enumerate}
    \item[(i)] $rz \geq d(z,x)$;
    \item[(ii)] $d(y,z) \geq \U(rz, d(y,x))$,
\end{enumerate} where conditions (i) and (ii) arise from $\V$-naturality of $r \mono d(-,x)$. A $\V$-coverage on $\mathcal{L}$ is, for each $x \in \mathcal{L}$, a collection $J(x)$ of sieves $r : \mathcal{L} \to [0,\infty]$ such that 
\begin{enumerate}
    \item[(iii)] the function $d(-,x) \in J(x)$;
    \item[(iv)] for any nonnegative rational number $q \geq \V(x,y)$ and $r \in J(x)$, the function $r_q$ defined by $$ r_q(z) := \max\{rz, \V(q,d(z,y))\}$$ is in $J(y)$.
\end{enumerate}
\end{example}

\section{Sieves under change of base} In this section, we prove several technical results that will allow us to relate collections of $\V$-enriched sieves with collections of $\U$-enriched sieves via base change. We build upon these results in later sections to prove a correspondence theorem for coverages. 

Establishing notation to be used for the rest of the paper, we consider categories $\U$ and $\V$ satisfying the hypotheses in \ref{assumptions-V}. We denote the unit objects in $\U, \V$ by $*_\U, *_\V$, the monoidal operation on both categories by $\otimes$, and fix generating families of finitely presentable objects $\G_\U$, $\G_\V$. 

We refer to a fixed lax monoidal functor $G : \V \to \U$, whose coherence morphisms we denote by $$ u : *_\U \to G(*_\V), \quad m_{ab} : G(a) \otimes G(b) \to G(a \otimes b).$$ Moreover, we take $G$ to be the right adjoint of the pair \ref{main-guy}, whose unit and counit we denote respectively by $\eta : \mathbbm{1}_\U \to GF$ and $\epsilon : FG \to \mathbbm{1}_\V$. For an enriched category $\X$ (over either $\U$ or $\V$), we will continue to denote composition in $\X_0$, as defined in \ref{underlying-category}, by $\cdot$ .

Before anything else, we need to establish that change of base via $G_*$ both preserves and reflects enriched natural families. The following proposition generalizes the observation made in \cite[1.3]{kelly} that if $\hom_\V(*_\V,-)$ is faithful, then $\V$-naturality of a family $\{\alpha_x : *_\V \to \C(Ax,Bx)\}$ is equivalent to ordinary naturality of ${\{(\alpha_x)_0 : A_0x \to B_0x\}}$.

\begin{proposition} \label{prop:G*-naturality} 
Let $\U$ and $\V$ be as above, and suppose $G$ is faithful. For $\V$-functors $A,B : \mathcal{C} \to \mathcal{D}$, the family $$\alpha := \{\alpha_x: *_\V \to \D(Ax,Bx)\}$$ is $\V$-natural if and only if the family $G_*\alpha$ is $\U$-natural.\end{proposition}

\begin{proof} Denote the left and right unitors in a monoidal category $\X$ by $\lambda_\X, \rho_\X$. If $\alpha$ is $\V$-natural, $\U$-naturality of $G_*\alpha$ follows from \cite[4.1.1]{cruttwell}. Conversely, suppose $G_*\alpha$ is $\U$-natural, so that \[\begin{tikzcd}
	& {G\C(x,y)} \\
	{G\D(Bx,By)} && {G\D(Ax,Ay)} \\
	{G\D(Bx,By) \otimes G(*_\V)} && {G(*_\V) \otimes G\D(Ax,Ay)} \\
	{G(\D(Bx,By) \otimes \D(Ax,Bx))} && {G(\D(Ay,By) \otimes \D(Ax,Ay))} \\
	& {G\D(Ax,By)}
	\arrow["{GB_{xy}}"', from=1-2, to=2-1]
	\arrow["{GA_{xy}}", from=1-2, to=2-3]
	\arrow["{(\id \otimes u) \cdot \rho_\U^{-1}}"', from=2-1, to=3-1]
	\arrow[""{name=0, anchor=center, inner sep=0}, "{(G_*\alpha)_x^*}"{description}, curve={height=-30pt}, from=2-1, to=5-2]
	\arrow["{(u \otimes \id) \cdot \lambda_\U^{-1}}", from=2-3, to=3-3]
	\arrow[""{name=1, anchor=center, inner sep=0}, "{[(G_*\alpha)_y]_*}"{description}, curve={height=30pt}, from=2-3, to=5-2]
	\arrow["{m \cdot (\id \otimes G\alpha_x)}"', from=3-1, to=4-1]
	\arrow["{m \cdot (G\alpha_y \otimes \id)}", from=3-3, to=4-3]
	\arrow["G\circ"', from=4-1, to=5-2]
	\arrow["G\circ", from=4-3, to=5-2]
	\arrow[shorten <=16pt, shorten >=21pt, Rightarrow, no head, from=0, to=4-1]
	\arrow[shorten <=16pt, shorten >=21pt, Rightarrow, no head, from=1, to=4-3]
\end{tikzcd}\] commutes. Suppressing subscripts, naturality of $m$ implies that $$ m \cdot (G\alpha \otimes \text{id}) = G(\alpha \otimes \text{id}) \cdot m, $$ so the above diagram becomes \[\begin{tikzcd}
	& {G\C(x,y)} \\
	{G\D(Bx,By)} && {G\D(Ax,Ay)} \\
	{G(\D(Bx,By) \otimes *_\V)} && {G(*_\V \otimes \D(Ax,Ay)} \\
	& {G\D(Ax,By)}
	\arrow["{GB_{xy}}"', from=1-2, to=2-1]
	\arrow["{GA_{xy}}", from=1-2, to=2-3]
	\arrow["{m \cdot (\id \otimes u) \cdot \rho_\U^{-1}}"', from=2-1, to=3-1]
	\arrow[""{name=0, anchor=center, inner sep=0}, "{(G_*\alpha)_x^*}"{description}, curve={height=-24pt}, from=2-1, to=4-2]
	\arrow["{m \cdot (u \otimes \id) \cdot \lambda_\U^{-1}}", from=2-3, to=3-3]
	\arrow[""{name=1, anchor=center, inner sep=0}, "{[(G_*\alpha)_y]_*}"{description}, curve={height=24pt}, from=2-3, to=4-2]
	\arrow["{G(\circ \cdot(\id \otimes \alpha_x))}"{description}, from=3-1, to=4-2]
	\arrow["{G(\circ \cdot(\alpha_y \otimes \id))}"{description}, from=3-3, to=4-2]
	\arrow[shorten <=14pt, Rightarrow, no head, from=0, to=3-1]
	\arrow[shorten <=14pt, Rightarrow, no head, from=1, to=3-3]
\end{tikzcd}.\] Finally, coherence of the monoidal functor $G$ means that we have $$ m \cdot(u \otimes \text{id}) \cdot \lambda_\U^{-1} = G\lambda_\V^{-1} \quad \text{and} \quad m \cdot (\text{id} \otimes u) \cdot \rho_\U^{-1} = G\rho_\V^{-1}, $$ so in fact both of the composites $$[(G_*\alpha)_x]_* \cdot GA_{xy} \qquad \text{ and } \qquad (G_*\alpha)_y^* \cdot GB_{xy} $$ are of the form $G(f)$ for some morphism $f$. We can therefore apply faithfulness of $G$, obtaining a commuting diagram \[\begin{tikzcd}
	{\C(x,y)} && {\D(Ax,Ay)} \\
	{\D(Bx,By)} && {*_\V \otimes \D(Ax,Ay)} \\
	{\D(Bx,By) \otimes *_\V} && {\D(Ax,By)}
	\arrow["{A_{xy}}", from=1-1, to=1-3]
	\arrow["{\lambda_\V^{-1}}", from=1-3, to=2-3]
	\arrow["{\rho_\V^{-1}}"', from=2-1, to=3-1]
	\arrow["{B_{xy}}"', from=1-1, to=2-1]
	\arrow["{\circ \cdot (\alpha_y \otimes \text{id})}", from=2-3, to=3-3]
	\arrow["{\circ \cdot (\text{id} \otimes \alpha_x)}"', from=3-1, to=3-3]
\end{tikzcd},\] which says exactly that $\alpha$ is $\V$-natural.
\end{proof}

Recalling the isomorphism $\D_0 \cong (G_*\D)_0$ of \ref{prop:c0-gc0-iso}, we have the following corollary, for later use:

\begin{corollary} \label{cor:G*-fullyfaithful-conservative} The 2-functor $G_*$ is locally fully faithful.
\end{corollary}

Our next purpose in this section will be to develop a method for taking a $\V$-presheaf and turning it into a $\U$-presheaf. Note that to start with a $\V$-functor $R : \C^\text{op} \to \V$ and obtain a $\U$-functor taking values in $\U$, it is not sufficient to change the base of $R$ using $G_*$, since in doing so, we obtain a $\U$-functor $G_*R$ taking values in $G_*\V$. If we want a $\U$-functor which takes values in $\U$, we need to do a little more. 

\begin{definition} \textbf{(Change of base for presheaves.)} \label{bcg-bcf} Let $\bcg$ be the unenriched functor defined as the composite \[\begin{tikzcd}
	{[\C^\text{op},\V]} && {[G_*\C^\text{op},G_*\V]} && {[G_*\C^\text{op},\U]}
	\arrow["{G_*}", from=1-1, to=1-3]
	\arrow["{G^\U \circ -}", from=1-3, to=1-5]
    \end{tikzcd},\] whose effect on objects is $$ R \; \longmapsto \; G^\U \circ (G_* R), $$ where $G^\U$ is as defined in \ref{eq:induced-U-adjunction}. Denoting the function of Remark \ref{remark:G0-is-GU} by $\bcg_{xy}$, we define the action on morphisms by $$ \{\alpha_x : *_\V \to \V(Ax,Bx)\} \longmapsto \{\bcg_{Ax,Bx}(\alpha_x) : *_\U \to \U(GAx, GBx)\}.$$
\end{definition}

To show that the effect of $\bcg$ on morphisms $\alpha$ is well-defined, we need the following lemma, together with Proposition \ref{prop:G*-naturality}.

\begin{lemma} \label{lemma:bcg-naturality}
    Suppose $G$ is the right adjoint of the pair \ref{main-guy}. For $\V$-presheaves ${A,B : \mathcal{C}^\text{op} \to \mathcal{V}}$, a family $$ \iota := \{ \iota_x : *_\U \to G\V(Ax,Bx) \}$$ is $\U$-natural if and only if the family $$ \bcg \iota := \{ G^\U_{Ax,Bx} \circ \iota_x : *_\U \to \U(GAx, GBx) \} $$ is $\U$-natural.
\end{lemma}

\begin{proof}
    For visual simplicity, we omit alphanumeric subscripts. Naturality of the counit $\epsilon$ for $F \dashv G$ implies that the top-right square in the diagram \[\begin{tikzcd}
	{G(\C(x,y))} & {G(\V(Bx,By))} & {G(\V(FGBx,By))} \\
	{G(\V(Ax,Ay))} & {G(\V(Ax,By))} & {G(\V(FGAx,By))} \\
	{G(\V(FGAx,Ay))} & {G(\V(FGAx,By))} & {G(\V(FGAx,By))}
	\arrow["{G_*B}", from=1-1, to=1-2]
	\arrow["{G_*A}"', from=1-1, to=2-1]
	\arrow["{G_*(\epsilon^*)}", from=1-2, to=1-3]
	\arrow["\iota"', from=1-2, to=2-2]
	\arrow["\iota"', from=1-3, to=2-3]
	\arrow["\iota", from=2-1, to=2-2]
	\arrow["{G_*(\epsilon^*)}"', from=2-1, to=3-1]
	\arrow["{G_*(\epsilon^*)}", from=2-2, to=2-3]
	\arrow["{G_*(\epsilon^*)}"', from=2-2, to=3-2]
	\arrow[shift right, no head, from=2-3, to=3-3]
	\arrow[no head, from=2-3, to=3-3]
	\arrow["\iota", from=3-1, to=3-2]
	\arrow[shift right, no head, from=3-2, to=3-3]
	\arrow[no head, from=3-2, to=3-3]
\end{tikzcd}\] commutes for any $x,y$, while commutativity of the bottom-left square follows from associativity of composition in $\V$. Thus commutativity of the outer square, expressing $\U$-naturality of $G_*(\epsilon^*) \cdot \iota$, is equivalent to commutativity of the upper-left square, expressing $\U$-naturality of $\iota$. Postcomposing each instance of $G_*(\epsilon^*)$ above with the appropriate component of $\Phi$ yields squares of the form $(\Phi \cdot \iota \cdot \Phi^{-1}) \cdot \Phi = \Phi \cdot \iota$, so commutativity of the diagram above is sufficient.
\end{proof}

Together with Proposition \ref{prop:G*-naturality} (and recalling the notation $k$ from \ref{eq:i-j-k}), this proves:

\begin{corollary} \label{cor:bcg-natural-iff-v-natural}
    For $\V$-presheaves ${A,B : \mathcal{C}^\text{op} \to \mathcal{V}}$, a family $$ \iota := \{ \iota_x : *_\V \to \V(Ax,Bx) \}$$ is $\V$-natural if and only if the family $$ \bcg \iota := \{ G^\U_{Ax,Bx} \circ k(\iota_x) : *_\U \to \U(GAx, GBx) \} $$ is $\U$-natural.
\end{corollary}

Our ultimate purpose in this section is to prove that the functor $\bcg$ defined above induces an injective map on posets of subobjects. To that end, we should hope that $\bcg$ preserves and reflects monomorphisms. In fact, we will prove that it preserves and reflects conical limits in the presheaf $\V$-category $[\C^\text{op},\V]$, for which we need a lemma.

We extend the results above to show that the change of base construction $\bcg$ for presheaves interacts nicely with conical limits, and thus in particular with monomorphisms, in the presheaf categories at hand.

\begin{proposition} \label{prop:bcg-limits}
    Suppose $G$ is faithful and conservative. The functor $$\bcg : [\C^\text{op},\V] \longrightarrow [G_*\C^\text{op}, \U]$$ is faithful and conservative, and both preserves and reflects conical limits.
\end{proposition}

\begin{proof} First, note that since $G^\U$ is a right $\U$-adjoint, so is $G^\U \circ (-)$. Corollary \ref{cor:G*-fullyfaithful-conservative} allows us to conclude that the composite $\bcg(-) = G^\U \circ G_*(-)$ preserves (unenriched) limits.

To show that $\bcg$ is faithful and conservative, we use the equality of Remark \ref{remark:G0-is-GU}, and recall the notation $i,j$ of \ref{eq:i-j-k} for the relevant isomorphisms of categories. If $\alpha, \beta : A \Rightarrow B$ are such that $\bcg \alpha = \bcg \beta$, then for each $x \in \C$, we have $j^{-1} G i(\alpha_x) = j^{-1}Gi(\beta_x)$. Since $j^{-1}$, $G$, and $i$ are all faithful, we have $\alpha_x = \beta_x$, whence $\alpha = \beta$. Similarly, since $j^{-1}$, $G$, and $i$ are all conservative, if each $j^{-1}Gi(\alpha_x)$ is an isomorphism, then so must $\alpha_x$ have been. Thus $\alpha$ is an isomorphism.

Since a conservative functor reflects any limits which it preserves, and recalling Remark \ref{remark:conicallimitsareordinarylimits}, which says that unenriched limits coincide with conical limits in our setting, we are done.
\end{proof}

We will later use Proposition \ref{prop:bcg-limits} in the specific case of cotensors, so we state this case now as a corollary.

\begin{corollary} \label{cor:bcg-cotensor} Suppose $G$ is the right adjoint of the pair \ref{main-guy}. For $y \in \G_\U$ and $R \in [\C^\text{op}, \V]$, $$\bcg \{Fy, R\} := G^\U \circ  G_*\{Fy,R\} \cong \{y, \bcg R\}.$$ \end{corollary}

\begin{proof}
        Cotensors in enriched functor categories can be realized as pointwise conical limits - see \cite[7.4.3]{riehl-cht}.
    \end{proof}

We have established our results on the limiting behavior of $\bcg$, so we turn to describing some of its lattice-theoretic properties. For an object $x$ of $\C$, recall that the set $\M_\V(\C(-,x))$ (defined in \ref{def:mono-lattice}) is a preorder under the relation $$ (\alpha : A \to \C(-,x)) \geq (\beta : B \to \C(-,x)) \; \iff \; \exists (\sigma : B \to A) \; \text{ such that } \beta = \alpha \sigma.$$ Note that $\M_\V(\C(-,x))$ may be large in general.

\begin{theorem}\label{thm:injective-monos} Suppose $G$ is faithful, conservative, and the right adjoint of the pair \ref{main-guy}. For any $\V$-functor $K : \C^\text{op} \to \V$, the assignment
        $$ \begin{tikzcd}
            {\M_{\V}(K)} & {\M_{\U}(\bcg K)}
            \arrow["{\bcg(-)}", from=1-1, to=1-2]
        \end{tikzcd} $$
is an injective morphism of preorders.
\end{theorem}

\begin{proof} Since $\bcg$ preserves limits by Proposition \ref{prop:bcg-limits}, the function is well-defined; that it is monotone is by functoriality of $\bcg$. Injectivity is by faithfulness of $\bcg$, also proved as part of Proposition \ref{prop:bcg-limits}.
\end{proof}

Observe that the map $\bcg(-)$ above, while well-defined on classes of monomorphisms, is not necessarily well-defined on subobjects, since we might in principle have a situation where $\bcg(\beta) = \bcg(\alpha) \circ \sigma$ for some isomorphism $\sigma$, even though $\alpha$ and $\beta$ are not isomorphic. Thus, in order to extend Theorem \ref{thm:injective-monos} to an injective mapping on subobjects, we need to ensure that $\bcg$ has one further property.

\begin{definition} \label{def:orom}
    Say that a functor $K$ is \textbf{order-reflecting on monomorphisms}, or simply \textbf{order-reflecting}, if it preserves monomorphisms and if whenever $f,g$ are monomorphisms such that $K(f) = K(g) \circ h$, then there exists some $h_0$ such that $h = K(h_0)$, and $f = g \circ h_0$.
\end{definition}

Clearly, any fully faithful right adjoint functor is order-reflecting, as is any functor which preserves monomorphisms and which is sourced in a category where all monomorphisms are strongly cartesian. In fact, to ensure that a right adjoint functor is order-reflecting, being faithful and conservative is sufficient. The following argument originated with \cite{fosco}, but we have modified it for our particular setting.

\begin{proposition} \label{prop:faithful-conservative-orom}
    If $G$ is faithful and conservative, then $G$ is order-reflecting on monomorphisms.
\end{proposition}

\begin{proof}
    Given monomorphisms $f,g$ in $\V$, we can take the pullback \[\begin{tikzcd}
	P & B \\
	A & C
	\arrow["r", from=1-1, to=1-2]
	\arrow["s"', from=1-1, to=2-1]
	\arrow["\lrcorner"{anchor=center, pos=0.125}, draw=none, from=1-1, to=2-2]
	\arrow["g", tail, from=1-2, to=2-2]
	\arrow["f"', tail, from=2-1, to=2-2]
\end{tikzcd}.\] Supposing we have $Gf = Gg \circ h$ for some $h : GA \to GB$, we have (since $G$ is a right adjoint) a pullback
\[\begin{tikzcd}
	GP & GB \\
	GA & GC
	\arrow["Gr", from=1-1, to=1-2]
	\arrow["Gs"', from=1-1, to=2-1]
	\arrow["\lrcorner"{anchor=center, pos=0.125}, draw=none, from=1-1, to=2-2]
	\arrow["Gg", rightarrowtail, from=1-2, to=2-2]
	\arrow["h"{description}, from=2-1, to=1-2]
	\arrow["Gf"', rightarrowtail, from=2-1, to=2-2]
\end{tikzcd}.\] It is easy to check that both triangles commute.

By the universal property of $GP$, there exists some $\nu : GA \to GP$ such that $Gs \circ \nu = \id_{GA}$ and $h = Gr \circ \nu$. Since $Gf = Gg \circ h$, we have $$Gg \circ Gr \circ \nu \circ Gs = Gg \circ Gr,$$ and since both $Gg$ and $Gr$ are monomorphisms, we have $\nu \circ Gs = \id_{GP}$. We see that $Gs$ is an isomorphism, and since $G$ is conservative, $s$ must have been an isomorphism in $\V$. Thus, we have $\nu = G(s^{-1})$. Taking $h_0 = r \circ s^{-1}$, we have $G(h_0) = h$, and since $G$ is faithful, we have $f = g \circ h_0$.
\end{proof}

To prove our main theorem, we will require $\bcg$ itself to be order-reflecting, a property which we will see is inherited from $G$. Note that $\bcg$ is not immediately order-reflecting as a result of Proposition \ref{prop:faithful-conservative-orom}, since we do not know that it is a right adjoint.

\begin{proposition} \label{prop:bcg-orom}
    If $G$ is order-reflecting on monomorphisms, then $\bcg$ is order-reflecting on monomorphisms.
\end{proposition}

\begin{proof}
    Suppose $\alpha, \beta : A \Rightarrow B$ are monomorphisms in $[\C^\text{op},\V]$ such that $\bcg \alpha = \bcg \beta \circ \sigma$, so that for each $x$, we have $\bcg \alpha_x = \bcg \beta_x \circ \sigma_x$. Referring to the equality in Remark \ref{remark:G0-is-GU}, we have $$  j^{-1}Gi(\alpha_x) = j^{-1}Gi(\beta_x) \circ \sigma_x $$ for each $x \in \C$. Since $j^{-1}$ is an isomorphism of categories, we have $\sigma_x = j^{-1}(\sigma_x')$ for some $\sigma_x'$, and thus $$Gi(\alpha_x) = Gi(\beta_x) \circ \sigma_x'. $$ Since $G$ is order-reflecting, and since by the same argument as for $j^{-1}$, $i$ is order-reflecting, there is some $(\sigma_0)_x$ such that $\alpha_x = \beta_x \circ (\sigma_0)_x$. Note that $\sigma_0 = j^{-1}Gi(\sigma) = \bcg(\sigma)$, so $\sigma_0$ is $\V$-natural by Corollary \ref{cor:bcg-natural-iff-v-natural}, and we have $\alpha = \beta \circ \sigma_0$.
\end{proof}

We come to the central objective of this section: proving that the functor $\bcg(-)$ induces an injective map on subobject lattices. Recall that as usual, for $K \in [\C^\text{op},\V]$, the preordering on $\M_\V(K)$ induces a partial order on $\text{Sub}_\V(K)$ (defined in \ref{def:sub-lattice}).

\begin{theorem} \label{thm:injective-sieves}
    Suppose $G$ is faithful and conservative. For a $\V$-functor ${K : \C^\text{op} \to \V}$, the assignment $$ \begin{tikzcd}
            {\text{Sub}_{\V}(K)} & {\text{Sub}_{\U}(\bcg K)}
            \arrow["{\bcg(-)}", from=1-1, to=1-2]
        \end{tikzcd} $$
induced by the morphism of preorders in Theorem \ref{thm:injective-monos} is an injective morphism of partially ordered sets.
\end{theorem}

\begin{proof}
    That the function is monotone follows from Theorem \ref{thm:injective-monos}, and it is well defined because $\bcg$ is order-reflecting. To see that it is injective, suppose $(r : R \mono K)$ and $(s : S \mono K)$ are such that $\bcg R = \bcg S$ as subobjects. In particular, there exists some $\U$-natural isomorphism $\sigma$ for which $\bcg r = \bcg s \circ \sigma$. Since $G$ is order-reflecting, $\bcg$ is order reflecting by Proposition \ref{prop:bcg-orom}. Thus there is some $\V$-natural transformation $\sigma_0$ for which $\bcg \sigma_0 = \sigma$, and since $\bcg$ is conservative by Proposition \ref{prop:bcg-limits}, $\sigma_0$ is an isomorphism. Then $r = s \circ \sigma_0$, so $R = S$ as subobjects.
\end{proof}

\subsection{Application to finite-length representations} Before moving on to our discussion of enriched coverages, we briefly mention an application of the above result. Here, we take the perspective that a $\V$-functor $\C^\text{op} \to \V$ is a $\V$-representation of $\C^\text{op}$ on $\V$, and obtain a generalization of a result from elementary representation theory.

As usual, we say that an ascending chain $$ R_0 \leq R_1 \leq ... \leq K $$ or descending chain $$ ... \leq R_1 \leq R_0 \leq K $$ of subobjects of a $\V$-functor $K$ \textbf{stabilizes} if there exists some $i$ such that for all $j \geq i$, we have $R_j = R_i$. We say that $K$ satisfies the \textbf{ascending chain condition} if any ascending chain of subobjects stabilizes, and the \textbf{descending chain condition} if any descending chain of subobjects stabilizes. We say that a $\V$-functor $K : \C^\text{op} \to \V$ has \textbf{finite length} if it satisfies both the ascending and descending chain conditions.

As an almost immediate consequence of Theorem \ref{thm:injective-sieves}, we have the following:

\begin{proposition} \label{prop:finitelength}
    Suppose $G$ is faithful and conservative, and say $Q : \C^\text{op} \to \V$ is a $\V$-functor. If $\bcg Q$ has finite length, then $Q$ has finite length. 
\end{proposition}

\begin{proof} First, note that $\bcg$ preserves ascending chains, since by \ref{thm:injective-sieves}, it is monotone. To see that $Q$ satisfies the ascending chain condition, let $$ R_0 \leq R_1 \leq ... \leq Q $$ be an ascending chain. Then $$ \bcg R_0 \leq \bcg R_1 \leq ... \leq Q$$ is an ascending chain, so there exists an $i$ such that $\bcg R_i = \bcg R_j$ whenever $j \geq i$. By injectivity of $\bcg$, we have $R_i = R_j$ as subobjects, so the original chain terminates. The proof is identical for the case of a descending chain.
\end{proof}

\section{Enriched coverages under change of base}

In this section, we extend Theorem \ref{thm:injective-sieves} to an analogous result for coverages. Below, we refer to the monoidal adjunction \ref{main-guy} of the previous sections.

\subsection{Lattices of enriched coverages.} Given a category $\mathcal{W}$ satisfying Hypothesis \ref{assumptions-V} and a $\mathcal{W}$-category $\X$, we establish some properties of the collection of $\mathcal{W}$-coverages on $\X$. If $\X$ is small and $\mathcal{W}$ is both complete and well-powered, as is true in the case where $\mathcal{W}$ satisfies \ref{assumptions-V}, then $[\X^\text{op}, \mathcal{W}]$ is well-powered, as proven in \cite[4.15]{BUNGE196964}. It follows that the collection of $\mathcal{W}$-coverages on $\X$, which we will denote by $\Sigma(\X,\mathcal{W})$, is a small set.

Exactly as for ordinary topologies on a set of points, as in \cite{larson1975lattice}, and Grothendieck topologies on an ordinary category, as in \cite[V3, 3.2.13]{borceux}, $\mathcal{W}$-coverages form a complete lattice.

Let $J,K$ be two $\mathcal{W}$-coverages on $\X$. We will say that $K$ is a \textbf{refinement} of $J$ (and $J$ is \textbf{coarser} than $K$) if $$ J(x) \subseteq K(x) $$ for all objects $x \in \X$, in which case we use the notation $J \subseteq K$. Say $J = K$ if $J(x) = K(x)$ for all $x$. It is routine to check that $\Sigma(\X,\mathcal{W})$ is partially ordered under refinement, with top element the $\W$-coverage $D$ defined by \begin{equation} \label{eq:discrete-topology}
    D(x) := \text{Sub}(\X(-,x))
\end{equation} and bottom element $I$, with $$ I(x) := \{\X(-,x)\}.$$ Moreover, given a family $\{J_\alpha\}_{\alpha \in A} \subset \Sigma(\X,\mathcal{W})$, the assignment \[ S(x) := \bigcap_\alpha J_\alpha(x) \] defines a $\mathcal{W}$-coverage, which is easily seen to be the finest one which is coarser than any of the $J_\alpha$. Using the fact that the greatest lower bound property implies the least upper bound property on a small set proves the following:

\begin{proposition} \label{gt-lattice} For $\X$ small and $\mathcal{W}$ satisfying \ref{assumptions-V}, the set $\Sigma(\X,\mathcal{W})$ of $\mathcal{W}$-coverages is a complete lattice.
\end{proposition}

\subsection{Change of base for $\V$-coverages.} To prove the following proposition, we want to be able to say that the left adjoint $F$ of \ref{main-guy} preserves generating families. This property is ensured if $G$ is faithful and conservative: it is proved in \cite[2.2.1]{borger-tholen} that $G$ is faithful and conservative if and only if the family $$\{Fx : x \in \mathcal{H}\}$$ is (extremally) generating in $\V$ whenever $\mathcal{H}$ is (extremally) generating in $\U$. By \cite[2.19]{adamek-rosicky}, $F$ preserves finitely presentable objects, and thus $$ F\U_{fp} := \{Fx : x \in \U_{fp}\} $$ is a generating set of finitely presentable objects in $\V$. Given $\G_\U$, we may therefore take $\G_\V = F\G_\U$.

In proving the following proposition, we will be concerned with generalized elements $f \in \hom_\V(\G_\V, \C(y,x))$ of the hom-objects of $\C$. By \cite[1.6]{bq} and the remarks above, it will suffice to restrict our attention to those of the form $Fg$ for some $g \in \G_\U$.

\begin{proposition} \label{prop:bcg-coverage} Suppose $G$ is faithful and conservative. For a $\V$-coverage $J$ on $\C$, the assignment to each object $x \in \C$ of the family $$\bcg J (x) = \{\bcg R \mid R \in J(x)\}$$ defines a $\U$-coverage on $G_*\C$.\end{proposition}

\begin{proof} We show that $\bcg J$ satisfies (T1) and (T2) of Definition \ref{v-gt}.

\textit{(T1).} For each $x \in \C$, we know that $\C(-,x) \in J(x)$, so by definition of $\bcg J$, we have $\bcg \C(-,x) \in \bcg J(x)$.

\textit{(T2).} Take any sieve $r: R \mono \C(-,x)$ in $J(x)$ (so that $(\bcg r: \bcg R \mono \bcg \C(-,x))$ is an arbitrary element of $\bcg J(x)$), any $g \in \G_\U$, and any generalized element ${a : g \to G\C(y,x)}$. We first show that the pullback $(\bcg R)_a$, given by the diagram \begin{equation} \label{eq:pullback-in-U-1}
    \begin{tikzcd}
		(\bcg R)_a \arrow[r] \arrow[d] & \{g,\bcg R\} \arrow[d, "\bcg r"] \\
		\bcg \C(-,y) \arrow[r, "a"] & \{g,\bcg \C(-,x)\}
	\end{tikzcd}
\end{equation} in $[G_*\C^\text{op},\U]$, is in $\bcg J(y)$. Taking the transpose $b : Fg \to \C(y,x)$ of $a$, we can form the pullback \[ \begin{tikzcd}
		R_b \arrow[r] \arrow[d] & \{Fg,R\} \arrow[d, "r"] \\
		\C(-,y) \arrow[r, "b"] & \{Fg,\C(-,x)\}
	\end{tikzcd} \] in $[\C^\text{op}, \V]$ of $r$ along $b$. Since $\bcg$ preserves limits by Proposition \ref{prop:bcg-limits}, the diagram \[ \begin{tikzcd}
	\bcg(R_b) \arrow[r] \arrow[d] & \bcg  \{Fg,R\} \arrow[d] \\
	\bcg \C(-,y) \arrow[r] & \bcg \{Fg,\C(-,x)\}
\end{tikzcd} \] is a pullback. Applying Corollary \ref{cor:bcg-cotensor} to the right-hand edge of this square, we have a pullback \begin{equation} \label{eq:pullback-in-U-2}
    \begin{tikzcd}
	\bcg(R_b) \arrow[r] \arrow[d] & \{g, \bcg R\} \arrow[d] \\
	\bcg \C(-,y) \arrow[r] & \{g, \bcg \C(-,x)\}
\end{tikzcd}.
\end{equation} Comparing the diagrams \ref{eq:pullback-in-U-1} and \ref{eq:pullback-in-U-2}, we see that $\bcg(R_b)$ and $(\bcg R)_a$ are pullbacks of the same diagram, so that $\bcg(R_b) \cong (\bcg R)_a$. Since $J$ is a $\V$-topology, we have $R_b \in J(y)$. By definition of $\bcg J$, we thus have $\bcg(R_b) \in \bcg J(y)$; and recalling that $\bcg J(y)$ was defined as a family of isomorphism classes of functors, we have $(\bcg R)_a \in \bcg J(y)$. Since $F\G_\U$ is a generating family for $\V$, and the result above holds for an arbitrary $g \in \G_\U$, we see that (T2) is satisfied. \end{proof}

We now have the machinery to prove the main result of this section, but before doing so, we need to address one minor technicality. Observe that since $\C$ is small, the $\U$-category $G_*\C$ is small. Since both $\U$ and $\V$ satisfy \ref{assumptions-V}, Proposition \ref{gt-lattice} shows that both $\Sigma(\C,\V)$ and $\Sigma(G_*\C,\U)$ are complete lattices, and so the statement of the theorem below makes sense.

\begin{theorem} \label{thm:injective-coverages} Suppose $G$ is faithful and conservative. The assignment $$ \begin{tikzcd}
            {\Sigma(\C,\V)} & {\Sigma(G_*\C,\U)}
            \arrow["{\bcg(-)}", from=1-1, to=1-2]
        \end{tikzcd} $$ is an injective morphism of lattices.
\end{theorem}

    \begin{proof} Proposition \ref{prop:bcg-coverage} shows that the assignment is well-defined. To see that it is monotone, suppose $J,K \in \Sigma(\C,\V)$ are such that $J \subseteq K$, so that $J(x) \subseteq K(x)$ for all objects $x$. Then for any $\bcg R \in \bcg J(x)$, we know that since $R \in K(x)$, we have $\bcg R \in \bcg K(x)$, whence $\bcg J \subseteq \bcg K$. 
    
    To see that meets are preserved, observe that \begin{align*} \bcg \left[\bigcap_\alpha J_\alpha \right](x) &= \{ \bcg R \mid R \in J_\alpha(x) \text{ for all } \alpha\} \\
    &= \bigcap_\alpha \{ \bcg R : R \in J_\alpha(x)\} \\
    &= \bigcap_\alpha \bcg J_\alpha(x). \end{align*}
    
    To prove injectivity, suppose $J,K$ are $\V$-coverages such that $\bcg J = \bcg K$. For all $x$, we thus have that (i) for each $\bcg R \in \bcg J(x)$, there exists an $S \in K(x)$ such that $\bcg R = \bcg S$; (ii) for each $\bcg S \in \bcg K(x)$, there exists an $R \in J(x)$ such that $\bcg S = \bcg R$. By \ref{thm:injective-sieves}, (i) implies that $J(x) \subseteq K(x)$, and (ii) implies that $K(x) \subseteq J(x)$, whence $J = K$. \end{proof}

\subsection{Examples of coverages under change of base} Having already discussed the case where the forgetful functor $\hom_\V(*_\V,-) : \V \longrightarrow \cat{Set}$ happens to be faithful, we give a few more examples of functors $G$ for which Theorem \ref{thm:injective-coverages} holds.

\begin{example} \textbf{(Monoids in $\W$)} As remarked in \cite[6.I]{measuring-comonoids}, with $\W$ satisfying \ref{assumptions-V}, the forgetful functor $G : \text{Mon}(\W) \to \W$ is monadic.
\end{example}

\begin{example} \textbf{(Restriction of scalars)} \label{example:restriction-of-scalars} Given a homomorphism $f : R \to S$ of commutative rings, restriction of scalars $$f^* : S\cat{Mod} \longrightarrow R\cat{Mod} $$ is lax monoidal and monadic, so Theorem \ref{thm:injective-coverages} says that any $S\cat{Mod}$-coverage on an $S$-linear category $\C$ induces a unique $R\cat{Mod}$-coverage on the corresponding $R$-linear category.
\end{example}

\begin{example} \textbf{(From strict 2-categories to simplicial categories)}
    Letting $\cat{Cat}$ denote a category of small strict categories, the nerve construction $$ N : \cat{Cat} \longrightarrow \cat{sSet} $$ is lax monoidal, fully faithful, and a right adjoint, so Theorem \ref{thm:injective-coverages} yields an injective mapping from $\cat{Cat}$-topologies on a strict 2-category $K$ to $\cat{sSet}$-topologies on $N_*K$.
\end{example}

\begin{example} \label{example:poset-to-proxet} \textbf{(From preorders to proximity sets)} Let ${\U =([0,1],\leq,\cdot,1)}$, where $\cdot$ denotes multiplication of real numbers, and let $\mathcal{L}$ be a $\U$-category. For a discussion of $\U$-categories in the context of formal concept analysis, we refer the reader to \cite{Pavlovic_2012}, where they are called proximity sets.

Recall that $\U(x,y) := \min\{1, y/x\}$, and that $\U$ is generated under filtered colimits by $\mathbb{Q} \cap [0,1]$, although it is not locally finitely presentable as a category. As in Example \ref{example:lawvere-coverage}, note that Definitions \ref{v-sieve} and \ref{v-gt} still make sense in this setting.

Unpacking Definition \ref{v-sieve} in this case, a $\U$-sieve on $x \in \mathcal{L}$ is a function ${r : \mathcal{L} \to [0,1]}$ which satisfies (for all $y,z \in \mathcal{L}$)
\begin{enumerate}
    \item[(i)] $\L(y,z) \cdot \L(x,y) \leq \U(rz, rx)$;
    \item[(ii)] $\U(rx,rx) = 0$;
    \item[(iii)] $\L(x,y) \leq \U(ry,rx)$;
    \item[(iv)] $rz \leq \L(z,x)$;
    \item[(v)] $\L(y,z) \leq \U(rz, \L(y,x))$,
\end{enumerate} where conditions (i)-(iii) arise from $\U$-functoriality of $r$, and conditions (iv) and (v) from $\U$-naturality of $r \mono \L(-,x)$. A $\U$-coverage on $\mathcal{L}$ is, for each $x \in \mathcal{L}$, a collection $J(x)$ of sieves $r : \mathcal{L} \to [0,1]$ such that 

\begin{enumerate}
    \item[(vi)] the function $\L(-,x) \in J(x)$;
    \item[(vii)] for any nonnegative rational number $q \leq \U(x,y)$ and $r \in J(x)$, the function $r_q$ defined by $$ r_q(z) := \min\{rz,\; \U(q,\L(z,y))\}$$ is in $J(y)$.
\end{enumerate}

The lax monoidal functor $G: \{0,1\} \hookrightarrow [0,1]$ which assigns $0 \mapsto 0$ and ${1 \mapsto 1}$ is easily seen to be both faithful and conservative, and is right adjoint to the functor $F: [0,1] \to \{0,1\}$ assigning $0 \mapsto 0$ and $x \mapsto 1$ whenever $x > 0$. As such, Theorem \ref{thm:injective-coverages} says that any $\{0,1\}$-coverage on a given poset $P$ (as in Example \ref{example:preorder-gt}) corresponds uniquely to a $[0,1]$-coverage on the proximity set $G_*P$.
\end{example} 

\begin{example} \textbf{(From proximity sets to Lawvere metric spaces, and back again)} \label{example:proxet-to-lawvere} We have an isomorphism $$ - \log(-) : [0,1] \leftrightarrows [0,\infty] : e^{-(-)}, $$ so a $[0,1]$-coverage on a proximity set $P$ corresponds uniquely to a $[0,\infty]$-coverage (Example \ref{example:lawvere-coverage}) on the Lawvere metric space corresponding to $P$ under $-\log(-)$. Similarly, any $[0,\infty]$-coverage uniquely determines a $[0,1]$-coverage.
\end{example}

\section{Enriched sheaves under change of base} In this section, we turn to investigating how change of base interacts with enriched sheaves, as defined in \cite{bq}. We begin by recalling the definition given in \cite[1.3]{bq} for an enriched sheaf. Below, we let $\W$ denote a category satisfying Hypothesis \ref{assumptions-V}, and $\X$ a small $\mathcal{W}$-category.

\begin{definition} \label{def:v-sheaf}
    A $\W$-functor $P \in [\X^\text{op},\W]$ is a \textbf{sheaf} for a $\W$-coverage $J$ when, given any object $x \in \X$, any $R \in J(x)$, any $g \in \G_\W$, and $\alpha$ such that
    \[\begin{tikzcd}
	R & {\X(-,x)} \\
	{\{g,P\}}
	\arrow["r", tail, from=1-1, to=1-2]
	\arrow["\alpha"', from=1-1, to=2-1]
	\arrow["{\exists ! \beta}", dotted, from=1-2, to=2-1]
\end{tikzcd},\] there exists a unique $\beta$ for which the diagram commutes. 
\end{definition}

Recall that a \textbf{localization} of $[\X^\text{op},\W]$ is a reflective $\W$-subcategory $\mathcal{K}$ whose reflector preserves finite weighted $\W$-limits. The central result of \cite{bq} says that, given a $\W$-topology $J$ on $\X$, we can construct a unique localization of $[\X^\text{op},\W]$, and vice-versa. In this section, we prove results concerning how change of base via $G$ interacts with this construction. Below, we suppose that $G : \V \to \U$ is faithful, conservative, and the right adjoint of the pair \ref{main-guy}.

\begin{definition} \label{def:sheafification} \cite[4.1, 4.4]{bq} Given a presheaf $P \in [\C^\text{op},\V]$, define a new presheaf $\Sigma P$ on objects by $$\Sigma P(x) = \text{colim}_{R \in J(x)} [\C^\text{op},\V](R,P).$$ The \textbf{sheafification} of $P$ with respect to $J$ is $\Sigma \Sigma P$. We will refer to the right adjoint $$ \ell : [\C^\text{op},\V] \longrightarrow \text{Sh}(\C,J) $$ to the inclusion functor $i: \text{Sh}_\V(\C,J) \hookrightarrow  [\C^\text{op},\V] $, where $\ell(P) := \Sigma \Sigma P$.
\end{definition}

A classical example is the case where $\V = \cat{Ab}$ and $\mathfrak{R}$ is a $\V$-topology on a $\V$-category $A$, as in Example \ref{example:ring-gabriel}. In this case, the functor $\ell$ is the canonical ring homomorphism from a ring $A$ to its localization $A_\mathfrak{R}$.

\begin{example} \cite[IX.1]{Stenstrom_1975} For a commutative ring $A$, we have $$A_\mathfrak{R} := \text{colim}_{I \in \mathfrak{R}} \hom_A(I, A/t(A)),$$ where $$t(A) := \{a \in A : aJ = 0 \text{ for some } J \in \mathfrak{R}\}.$$ In particular, if $S$ is a multiplicatively closed subset of $A$ containing no zero divisors and such that for $s \in S$ and $a \in A$, there exist $t \in S$ and $b \in A$ such that $sb = at$, the family $$ \mathfrak{R} := \{I \triangleleft A : I \cap S \neq \nada\}$$ (where $I \triangleleft A$ means that $I$ is an ideal of $A$) defines a Gabriel topology on $A$, and $A_\mathfrak{R}$ is isomorphic to the ring of fractions $A[S^{-1}]$.
\end{example}

We can also sheafify $\U$-presheaves on $G_*\C$ with respect to the $\U$-coverage $\bcg J$ of Proposition \ref{prop:bcg-coverage}. We will use the notation $$ \ell_G \dashv i_G : \text{Sh}(G_*\C,\bcg J) \leftrightarrows [G_*\C^\text{op}, \U] $$ for the resulting localization, and denote the units of both adjunctions $i \dashv \ell$ and $i_G \dashv \ell_G$ by $\eta$. 

It seems natural to ask whether sheafification `commutes' with change of base, in the sense that $\bcg (i\ell P) \cong i_G\ell_G(\bcg P)$ as sheaves. We will see that in the case where $G$ is only faithful and conservative, we at least obtain a distinguished morphism $\bcg( i \ell P) \to i_G \ell_G (\bcg P)$; however, if $G$ is full, the isomorphism is guaranteed.

\begin{proposition} \label{prop:bcg-sheaf} Let $J$ be a $\V$-coverage on $\C$ and $P \in [\C^\text{op},\V]$ be a sheaf for $J$. If $G$ is full, then $\bcg P$ is a sheaf for $\bcg J$. 
\end{proposition} 

    \begin{proof} Say $P \in \text{Sh}_\V(\C,J)$, and suppose that $\gamma : \bcg \C(-,x) \to \{y, \bcg P\}$, $r : R \mono \C(-,x)$, $g \in \G_\V$ and $\alpha : R \to \{g,P\}$ are such that $$\bcg \alpha = \gamma \circ \bcg r.$$ 
    
    By Definition \ref{def:v-sheaf}, there exists a unique $\beta : \C(-,x) \to \{g,P\}$ for which $$ \gamma_y Gr_y = G\beta_y \circ Gr_y = G\alpha_y $$ for each object $y \in \C$. Since $G$ is full, $\gamma_y$ has the form $G\delta_y$ for some ${\delta_y : \C(y,x) \to \{Fg, Py\}}$. Since $G$ is faithful, uniqueness of $\beta$ implies that $\delta_y = \beta_y$, whence $\gamma = \bcg \beta$.
    \end{proof}

Given $S \in [\C^\text{op},\V]$ and $r : R \mono S$, define $\widehat{R}$ to be the pullback \[\begin{tikzcd}
	{\widehat{R}} & {i \ell(R)} \\
	S & {i \ell(S)}
	\arrow[from=1-1, to=1-2]
	\arrow[from=1-1, to=2-1]
	\arrow["{i \ell(r)}", from=1-2, to=2-2]
	\arrow["{\eta_S}", from=2-1, to=2-2]
\end{tikzcd}.\] The operation $R \mapsto \widehat{R}$ is a universal closure operation on $[\C^\text{op},\V]$ in the sense of \cite[1.4]{bq}. A presheaf $R$ is called \textbf{dense} if $\widehat{R} = S$. 

For visual simplicity, we define $$ \widetilde{\eta_Q}:=\bcg(\eta_Q)  \quad \text{ and } \quad \eta_{\widetilde{Q}} := \eta_{\bcg Q} .$$

\begin{theorem} Suppose $G$ is faithful and conservative. For $P \in [\C^\text{op},\V]$, the unit $$\eta_{\widetilde{P}} : \bcg P \to i_G \ell_G(\bcg P)$$ factors uniquely through $\bcg (i \ell P)$; and if $G$ is full, $\bcg (i\ell P) \cong i_G\ell_G(\bcg P)$.

    \begin{proof}
        Since $i$ is fully faithful, we have for any $Q \in [\C^\text{op},\V]$ that the unit ${\eta_Q : Q \to i \ell Q}$ is an isomorphism. Then $i \ell(\eta_Q)$ is an isomorphism, and since isomorphisms are pullback stable, we have $\widehat{Q} \cong i \ell Q$; in other words, $\eta_Q$ is dense. Since $\bcg$ preserves conical limits, we have $$\bcg \widehat{Q} \cong \widehat{\bcg Q} \cong \bcg( i \ell Q ),$$ so that $\widetilde{\eta_Q}$ is dense.

        The result \cite[2.2]{bq} says that $P$ is (isomorphic to) a sheaf for $J$ exactly when, for every dense monomorphism $r : R \mono Q$ and morphism $s: R \to P$, there is a unique $t : Q \to P$ for which $r = ts$. In particular, since $i_G \ell_G (\bcg P)$ is a sheaf for $\bcg J$ and $\widetilde{\eta_P}: \bcg P \to \bcg(i \ell P)$ is dense, there is a unique morphism $\tau$ for which \[\begin{tikzcd}
	{\bcg P} & {i_G \ell_G (\bcg P)} \\
	{\bcg(i \ell P)}
	\arrow["{\eta_{\widetilde{P}}}", from=1-1, to=1-2]
	\arrow["{\widetilde{\eta_P}}"', tail, from=1-1, to=2-1]
	\arrow["\tau"', dashed, from=2-1, to=1-2]
        \end{tikzcd}\] commutes. If $G$ is full, Proposition \ref{prop:bcg-sheaf} says that $\bcg(i \ell P)$ is a sheaf for $\bcg J$, so the same argument yields a unique factorization of $\widetilde{\eta_P}$ through $\eta_{\widetilde{P}}$, say $\sigma \eta_{\widetilde{P}} = \widetilde{\eta_P}$. We then have, for example, $$\tau \sigma \eta_{\widetilde{P}} = \tau \widetilde{\eta_P} = \eta_{\widetilde{P}},$$ so since $\eta_{\widetilde{P}}$ is an isomorphism, $\tau \sigma$ is an identity. The same argument shows that $\sigma \tau$ is an identity, so we have $\bcg(i \ell P) \cong i_G \ell_G(\bcg P)$.
    \end{proof}
\end{theorem}

\section{Gabriel topologies}

Our goal in this section is to illustrate via an example (namely \ref{bad-gabriel}) that the conclusion of Theorem \ref{thm:injective-coverages} may fail if the functor $G : \V \to \U$ is not faithful. Toward that end, we generalize Definition \ref{example:ring-gabriel} of a Gabriel topology on a ring---that is, on a monoid object in $\cat{Ab}$---to monoid objects in an arbitrary $\V$ satisfying \ref{assumptions-V}.

Perhaps among the easiest $\V$-categories to understand are one-object $\V$-categories, which are easily seen to coincide with the monoid objects in $\V$---that is to say, those objects $A$ of $\V$ equipped with suitably coherent morphisms $m : A \otimes A \to A$ and $\nu : *_{\V} \to A$. Denoting the opposite monoid of $A$ by $A^\text{op}$, we can use any such $A$ to define a \textbf{right $A$-module} in $\V$: an object $M$ of $\V$ equipped with a morphism $$ \psi: A^{\text{op}} \otimes M \to M, $$ called a \textbf{right $A$-action} on $M$, satisfying coherence conditions encoding associativity and unitality of the action. (For brevity, we do not discuss coherence in detail; the uninitiated reader may consult \cite[VII.3-4]{cwm}.) In particular, a monoid object $(A, m, \nu)$ of $\V$ is always a right module over itself. To emphasize that we are viewing $A$ as a right $A$-module, we will sometimes use the notation $A_A$. By an \textbf{$A$-submodule} of $M$, we mean an $A$-module $N$ admitting a monomorphism $\iota : N \mono M$ in $\V$, and whose $A$-action is compatible with that of $M$ in a sense that we will make precise below. 

When $\V$ is closed monoidal, as in the present setting, we can `transpose' a right action and its requisite coherence diagrams, obtaining a morphism $$ \phi : A^\text{op} \to \V(M,M)$$ in $\V_0$ which satisfies conditions encoding compatibility of the monoidal structure on $A^\text{op}$ with the composition and identities in $\V$. If we shift our perspective and view $A^\text{op}$ as a $\V$-category with a single object $\bullet$, the coherence of $\phi$ expresses $\V$-functoriality of the assignment $\bullet \mapsto M$. From this perspective, $\V$-sieves have straightforward descriptions in terms of subobjects of $A$.

\begin{proposition} \label{v-ideal} If $\V$ is closed monoidal and $\A$ is a one-object $\V$-category with $\A(\bullet,\bullet) = A \in \text{Mon}(\V)$, a $\V$-sieve on $\bullet$ is equivalently an $A$-submodule of $A_A$.
\end{proposition}

\begin{proof} We unpack the definition of a subfunctor $\mathcal{I}(-)$ of $\A(-,\bullet) : \A^\text{op} \to \V$. Say $\mathcal{I}(-) : \A^\text{op} \to \V$ sends $\bullet \mapsto I$, and let $\phi : \A^\text{op}(\bullet,\bullet) = A^\text{op} \to \V(I,I)$ be the hom-component of $\mathcal{I}(-)$. Functoriality of $\mathcal{I}(-)$ says that the diagrams \[\begin{tikzcd}
	{A^\text{op} \otimes A^\text{op}} & {A^\text{op}} && {*_\V} & {A^\text{op}} \\
	{\V(I,I) \otimes \V(I,I) } & {\V(I,I) } &&& {\V(I,I) }
	\arrow["m", from=1-1, to=1-2]
	\arrow["{\phi \otimes \phi}"', from=1-1, to=2-1]
	\arrow["\phi", from=1-2, to=2-2]
	\arrow["\nu", from=1-4, to=1-5]
	\arrow["{\text{id}}"', from=1-4, to=2-5]
	\arrow["\phi", from=1-5, to=2-5]
	\arrow["\circ"', from=2-1, to=2-2]
\end{tikzcd}\] commute. Denoting the transpose of $\phi$ by $\psi : A^\text{op} \otimes I \to I$, commutativity of the diagrams above is equivalent to commutativity of \[\begin{tikzcd}
	{A^\text{op} \otimes I} && I & {*_\V \otimes I} & {A^\text{op} \otimes I} \\
	{(A^\text{op} \otimes A^\text{op}) \otimes I} && I && I \\
	{(\V(I,I) \otimes \V(I,I)) \otimes I} && I
	\arrow["\psi", from=1-1, to=1-3]
	\arrow[shift right, no head, from=1-3, to=2-3]
	\arrow["{\nu \otimes \text{id}}", from=1-4, to=1-5]
	\arrow["{\lambda^{-1}}"', from=1-4, to=2-5]
	\arrow["\psi", from=1-5, to=2-5]
	\arrow["{m \otimes \text{id}}", from=2-1, to=1-1]
	\arrow["h", from=2-1, to=2-3]
	\arrow["{(\psi \otimes \psi) \otimes \text{id}}"', from=2-1, to=3-1]
	\arrow[no head, from=2-3, to=1-3]
	\arrow[shift right, no head, from=2-3, to=3-3]
	\arrow[no head, from=2-3, to=3-3]
	\arrow["{\circ^\flat}"', from=3-1, to=3-3]
\end{tikzcd},\] where $h = \psi(1 \otimes \psi)\alpha$, and with $\alpha$ and $\lambda$ respectively denoting the associator and left-unitor in $\V$. Commutativity of the top square in the left-hand diagram above is equivalent to associativity of $\psi$ as a right action of $A$ on $I$, and the triangle is equivalent to unitality. We see that $I$ is a right $A$-module.

Having a $\V$-natural transformation $\iota : \mathcal{I}(-) \Rightarrow \A(-,\bullet)$ with monic components says that we have a monomorphism $I \mono A$ in $\V_0$ which satisfies \[\begin{tikzcd}
	{A^\text{op}} & {\V(I,I)} \\
	{\V(A,A)} & {\V(I,A)}
	\arrow["{\iota_*}", from=1-2, to=2-2]
	\arrow["{\iota^*}"', from=2-1, to=2-2]
	\arrow["\phi", from=1-1, to=1-2]
	\arrow["m^\flat"', from=1-1, to=2-1]
\end{tikzcd},\] expressing compatibility of the right $A$-action on $I$ with the right $A$-action of $A$ on itself. 

In the converse direction, say given a right $A$-submodule $I$ of $A_A$, it is easy to check (by showing that commutativity is satisfied in the diagrams above) that $\bullet \mapsto I$ determines a $\V$-subfunctor of $\A(-,\bullet)$.
\end{proof}

Pullbacks of sieves on $\bullet \in \A$, as in \ref{v-gt} (T2), are somewhat simpler to describe than in the general case. Given $f : G \to \A(\bullet,\bullet) = A$, $f$ induces a morphism $$ G \to [\A^\text{op},\V](\A(-,\bullet),\A(-,\bullet)) $$ by the enriched Yoneda lemma, and thus a morphism $$\A(-,\bullet) \to \{G,\A(-,\bullet)\}.$$ Let $\iota: \mathcal{I}(-) \mono \A(-,\bullet)$. Since $\A$ has only one object, the pullback of the diagram \[\begin{tikzcd}
	{\A(-,\bullet)} && {\{G,\A(-,\bullet)\}} && {\{G, \mathcal{I}(-)\}}
	\arrow["f", from=1-1, to=1-3]
	\arrow["\iota"', hook', from=1-5, to=1-3]
\end{tikzcd}\] in $[\A^\text{op},\V]$ is uniquely determined by the pullback
\begin{equation} \label{little-pback}
\begin{tikzcd}
	A && {\V(G,A)} && {\V(G,I)}
	\arrow["f", from=1-1, to=1-3]
	\arrow["\iota"', hook', from=1-5, to=1-3]
\end{tikzcd}\end{equation} in $\V$. In the case where $\A$ has only one object, we identify the pullback $I_f$ in the functor category with the pullback of the diagram \ref{little-pback} in $\V$.

In light of the discussion above, we see that Example \ref{example:ring-gabriel} is the case $\V = \cat{Ab}$ of the following:

\begin{definition} \label{v-gabriel} Given a monoid object $A$ of $\V$, a \textbf{(right) $\V$-Gabriel topology} on $A$ is a non-empty family $\mathfrak{R}$ of right $A$-submodules of $A_A$ such that

\begin{itemize}
		\item[(V1)] if $I \in \mathfrak{R}$ and $J$ is a right $A$-submodule of $A_A$ such that $I$ is a right $A$-submodule of $J$, then $J \in \mathfrak{R}$;
		
		\item[(V2)] for any $(\iota : I \mono A) \in \mathfrak{R}$, $G \in \V_{fp}$, and $f : G \to A$ in $\V_0$, the pullback $I_f$ of the diagram \ref{little-pback} is in $\mathfrak{R}$;
		
		\item[(V3)] if $I \in \mathfrak{R}$ and $J$ is a right $A$-submodule of $A_A$ such that $J_f \in \mathfrak{R}$ for all $f : G \to I$, then $J \in \mathfrak{R}$.
	\end{itemize}

    \end{definition} 
    
    Squinting at \ref{v-gabriel}, the reader might guess that the following is true, although it may not be obviously apparent that (V1) is a perfect analogue of (T1) in \ref{v-gt}. We provide a bit more detail:

\begin{proposition} \label{v-topology-v-gabriel}
    Let $A \in \text{Mon}(\V)$, and let $\mathfrak{R}$ be a set of right $A$-submodules of $A_A$. Denote by $\mathcal{A}$ the one-object $\V$-category with $\A(\bullet,\bullet) = A$. Given a right $A$-submodule $I \mono A$, denote the $\V$-subfunctor $\bullet \mapsto I$ of $\A(-,\bullet)$ by $\mathcal{I}(-)$. The following are equivalent:

    \begin{itemize}
        \item[(i)] $\mathfrak{R}$ is a $\V$-Gabriel topology on $A$;
        \item[(ii)] $\Tau := \{ \mathcal{I}(-) : I \in \mathfrak{R}\}$ is a $\V$-topology on $\A$.
    \end{itemize}

    \begin{proof} That (T2) and (T3) are respectively equivalent to (V2) and (V3) follows directly from the definitions \ref{v-ideal} and \ref{little-pback}. Moreover if (V1) holds for $\mathfrak{R}$, the fact that $\mathfrak{R}$ is nonempty immediately implies (T1).

    The only subtlety is in proving that (V1) holds for $\mathfrak{R}$, given (ii). Suppose that $I \in \mathfrak{R}$ is such that $\iota : I \to A$ factors as \[\begin{tikzcd}
	I & J & A
	\arrow["i", tail, from=1-1, to=1-2]
	\arrow["j", tail, from=1-2, to=1-3]
    \end{tikzcd}\] for some $A$-submodule $J$ of $A_A$. If $f : G \to I$ has $G \in \V_{fp}$, then $\iota f = jif : A \to \V(G,A)$, so that the pullback $\mathcal{J}_f$ of \[\begin{tikzcd}
	{\A(-,\bullet)} & {\{G,\A(-,\bullet)\}} & {\{G,\mathcal{J}(-)\}}
	\arrow["{\iota f = jif}", from=1-1, to=1-2]
	\arrow["{j_*}"', from=1-3, to=1-2]
\end{tikzcd}\] is $\mathcal{A}(-,\bullet) \in \Tau$. Since $\Tau$ is a $\V$-topology, we have $\mathcal{J}(-) \in \Tau$, so that $J \in \mathfrak{R}$.
    \end{proof}
\end{proposition}

We briefly mention an application of Theorem \ref{thm:injective-coverages} in this setting. The discussion at the beginning of this section shows that if $\A$ is a one-object $\V$-category, the $\V$-presheaves on $\A$ are exactly the right $\A$-module objects in $\V$, so that $[\A^\text{op},\V] = \text{Mod-}\A$. Theorem \cite[1.5]{bq} says that there is a bijection between reflective subcategories of $\text{Mod-}\A$ and $\V$-topologies on $\A$. If $G : \V \to \U$ is a faithful, conservative right adjoint, Theorem \ref{thm:injective-coverages} says that any reflective subcategory of $\text{Mod-}\A$ corresponds uniquely to a reflective subcategory of $G_*\A$.

\begin{example}
    Consider the case where $G$ is restriction of scalars along a ring homomorphism $f : R \to S$ (Example \ref{example:restriction-of-scalars}) and $\A$ is an $S$-algebra, so that $G_*\A$ is simply $\A$ viewed as an $R$-algebra. In \cite[VI.4.2]{Stenstrom_1975}, reflective subcategories of $\text{Mod-}\A$ are identified with their reflectors, there referred to as left exact preradical functors on $\text{Mod-}\A$. The argument above shows that any left exact preradical on $\text{Mod-}\A$ corresponds uniquely to a left exact preradical on $\text{Mod-}G_*\A$.
\end{example}

\subsection{Graded Gabriel topologies on a graded algebra}

For the rest of this section, we consider a field $k$, and set $\V = \cat{grMod}_k$, the category of $\Z$-graded $k$-modules. Recall that the monoidal unit in $\V$ is $k$, viewed as a $\Z$-graded algebra concentrated in degree 0, and the internal hom in $\V$ is defined as $$ \V(M,N) := \bigoplus_{i \in \Z} \hom_i(M,N),$$ where $\hom_i(M,N)$ denotes the collection of $k$-module homomorphisms $f$ for which $f(M_j) \subset N_{j+i}$, which we call \textbf{morphisms of degree $i$}. Uninitiated readers can find a detailed treatment of graded algebras in \cite{nastasescu-grt}.

The functor $$\hom_\V(k,-) : \V \to \cat{Set}$$ has a left adjoint $k [-]$ in $\cat{Cat}$ which takes a set $X$ to the free graded $k$-module $k[X]$ generated in degree 0 by the elements of $X$. Since the functor $\hom_\V(k,-)$ is lax and the functor $k[-]$ is strong monoidal, they comprise an adjunction in $\mathbf{MonCat}_\ell$  by \cite[1.5]{doctrinal-adjunction}. We will see that $k [-] \dashv \hom_\V(k,-)$ yields an example where the assignment $\bcg(-)$ of Theorem \ref{thm:injective-coverages} is not injective.

\begin{example} \label{notfaithful} $\hom_\V(k,-) : \V \to \cat{Set}$ is not faithful - to see this, take any two distinct graded $k$-modules, say $M$ and $N$, with $M_0 = N_0 = 0$, and recall that $$\hom_\V(k,M) \cong \hom_{k}(k,M_0) \cong \{0\} $$ (and similarly for $N$). As long as there exists a non-trivial graded module homomorphism $M \to N$, for example, in the case of $M$ and $N$ with homogeneous components defined by $$ M_i = \begin{cases} 0 & i < 2 \\  k & i \geq 2 \end{cases}, \quad N_i = \begin{cases} 0 & i < 1 \\ k & i \geq 1 \end{cases}, $$ the map $$\V(M,N) \to \cat{Set}(\hom_\V(k,M), \hom_\V(k,N)) \cong \{0\}$$ is not injective.
\end{example}

Below, we construct an example of two $\V$-coverages which correspond to the same $\cat{Set}$-coverage under change of base, toward which our first task is to describe $\V$-sieves and their pullbacks. 

First, we refresh some terminology: A \textbf{homogeneous element} of a graded ring $A = \bigoplus_{i \in \Z} A_i$  is simply an element of $A_i$ for some $i$, the set of all such we will denote by $h(A)$. Recall that a left or right ideal $I$ of $A$ is called \textbf{homogeneous} if whenever $\sum a_i \in I$, each homogeneous element $a_i \in A_i$ in the sum is itself an element of $I$; or equivalently, if $I$ is a graded $A$-submodule of $A$. 

As a corollary to \ref{v-ideal}, we have the following:

\begin{corollary} \label{grmod-sieve} Given $A \in \cat{grAlg}_k$, viewed as a $\cat{grMod}_k$-category with one object $\bullet$, the $\V$-sieves on $\bullet$ are exactly the homogeneous right ideals of $A$. \end{corollary}

As described in \cite[p. 21]{nastasescu-grt}, $\V$ admits a dense generating family: For $i \in \Z$, define the homogeneous components of a graded $k$-module $k(-i)$ by $$ k(-i)_j := k_{j-i}, $$ so that $$ k(-i)_{i} = k_0 = k, $$ and $k(-i)_j = 0$ otherwise. Any graded $k$-module is the filtered colimit of its finite-dimensional graded subspaces, and any finite-dimensional graded $k$-module is isomorphic to the direct sum of the objects in $\{k(-j)\}_{j \in J}$ for some $J \subset \Z$. Thus, to construct the pullback as in \ref{v-gabriel} (V2), we need only consider pullbacks along graded module morphisms $f : k(-i) \to A$. 

Given a morphism $f : k(-i) \to A$ of graded $k$-modules and a homogeneous right ideal $I \subset A$, the pullback of the diagram \begin{equation} \label{grmod-pullback} \begin{tikzcd}
	{A} && {\V(k(-i),A)} && {\V(k(-i),I)}
	\arrow["f", from=1-1, to=1-3]
	\arrow[rightarrow, hook', "\text{inc}"', from=1-5, to=1-3]
\end{tikzcd}, \end{equation} where $f$ is identified with the map $1_A \mapsto f(1_k)$, is the homogeneous right ideal $(I : f(1_k))$.

With \ref{grmod-sieve} and \ref{grmod-pullback} in hand, we can define an analogue of \ref{example:ring-gabriel} for graded $k$-algebras.

\begin{definition} \label{graded-gabriel} A \textbf{graded (right) Gabriel topology} on $A$ is a non-empty set $\mathfrak{R}$ of homogeneous right ideals of $A$ satisfying

\begin{enumerate}
    \item[(G1)] if $I \in \mathfrak{R}$ and $J$ is a homogeneous right ideal of $A$ for which $I \subset J$, then $J \in \mathfrak{R}$;

    \item[(G2)] if $I \in \mathfrak{R}$, then $(I : x) \in \mathfrak{R}$ for all $x \in h(A)$;

    \item[(G3)] if $I \in \mathfrak{R}$ and $J$ is a homogeneous right ideal of $A$ such that $(J : x) \in \mathfrak{R}$ for all $x \in h(I)$, then $J \in \mathfrak{R}$.
\end{enumerate}

\end{definition}
    
Any multiplicatively closed set $S$ of homogeneous elements of $A$ gives rise to a graded Gabriel topology by letting $H_S$ be the collection of homogeneous right ideals defined by $$ H_S := \{I \mid (I : a) \cap S \neq \nada \; \text{for all homogeneous elements } a \in A\}, $$ as in \cite[II.9.11]{nastasescu-grt}.

    \begin{example} \label{bad-gabriel}
        For a field $k$, take $A$ to be the commutative ring $k[x, y]$, graded by polynomial degree. Set $$ S := \{1, x, x^2,...\}  \text{ and } T := \{1, y, y^2,...\}, $$ and consider the change of base given by $$ G = \hom_\V(k,-) : \V \to \cat{Set}. $$

        The families $S$ and $T$ generate distinct $\V$-coverages on $A$, namely $$ H_S = \{ I \triangleleft A : I \text{ is homogeneous and } x^n \in I \text{ for some } n\}$$ and $$H_T = \{I \triangleleft A : I \text{ is homogeneous and } y^n \in I \text{ for some } n\},$$ where the notation $I \triangleleft A$ means $I$ is an ideal of $A$.
        
        On the other hand, given any $M \in \V$, we have $$ \hom_\V(k, M) \cong \hom_k(k,M_0) \cong M_0,$$ so in particular, we have $\bcg I \cong \hom_k(k,I_0)$ for any $I$ in $H_S$ or $H_T$. Recall that the degree-0 elements of $A$ are exactly the scalars $k$; thus, if $I \neq A$, we have $I_0 = \{0\}$ (otherwise $I$ contains a unit of $A$), and if $I = A$, we have $I_0 = k$. Then $$\bcg H_S = \bcg H_T = \{k, (0)\}.$$ We see that the conclusion of Theorem \ref{thm:injective-coverages} fails in this case.
        
    \end{example}

\bibliographystyle{amsplain}
\bibliography{refs}

\end{spacing}

\end{document}